\documentclass[11pt]{amsart}

\usepackage{amsmath,amssymb,amsthm}
\usepackage[T1]{fontenc}
\usepackage{hyperref}
\usepackage{tikz}
\usepackage{tikz-cd}
\usetikzlibrary{arrows.meta}
\hypersetup{hidelinks}
\usepackage[colorinlistoftodos]{todonotes}

\theoremstyle{plain}
\newtheorem{prop}{Proposition}
\newtheorem{theorem}{Theorem}

\newtheorem*{theorem*}{Theorem}
\newtheorem*{cor*}{Corollary}

\theoremstyle{definition}
\newtheorem{definition}{Definition}

\theoremstyle{remark}
\newtheorem{rem}{Remark}

\newcommand{\var}{\operatorname{var}}
\newcommand{\can}{\operatorname{can}}
\newcommand{\Var}{\operatorname{Var}}

\newcommand{\Q}{\mathbb Q}
\newcommand{\C}{\mathbb C}
\newcommand{\PP}{\mathbb P^1}
\newcommand{\Gm}{\mathbb G_m}
\newcommand{\DR}{\operatorname{DR}}
\newcommand{\Ass}{A}

\title[The Drinfeld associator and $A_\infty$ functors]
{The Drinfeld associator and $A_\infty$ functors}
\author{Nikita Markarian}
\date{}

\begin{document}
\begin{abstract}
We compare two mixed Hodge structures attached to a unipotent variation $\mathbb V$ of mixed Hodge--Tate structures on
$X=\mathbb P^1\setminus{0,1,\infty}$: the cohomology of its extension by $*$ at $0$ and by $!$ at $\infty$, and the vanishing cycles at $1$. The semi-holonomy isomorphism constructed in \cite{MarkarianConvolution} identifies these two structures on their underlying rational structures, compatibly with their weight filtrations, but need not respect their Hodge filtrations. We compute the resulting Hodge defect. Its universal noncommutative series is given by a specific component of the inverse Drinfeld associator multiplied by an explicit local monodromy factor. We conclude with a formal $A_\infty$ interpretation of the coefficients of the Hodge defect in terms of weight-framed Tate extensions.
\end{abstract}

\email{nikita.markarian@gmail.com}
\address{UMR 7501, Université de Strasbourg,
7 rue René Descartes,
67084 Strasbourg Cedex, France}

\maketitle

\section*{Introduction}

We study the Drinfeld associator through the action of cohomological
functors on iterated extensions of Tate objects.  The guiding
question is how a functor acts on successive extensions once its
values on the Tate objects are known.  We compare cohomology and
vanishing cycles, calculate their relative periods, and identify
the resulting series in terms of the associator.

Let $\mathbb V$ be an admissible unipotent variation of mixed
Hodge--Tate structures on
$X=\PP\setminus\{0,1,\infty\}$, with underlying rational local
system $V$.  The semi-holonomy construction of
\cite{MarkarianConvolution} gives an isomorphism
\[
\varphi_I:\Phi(j_{**!}V)\xrightarrow{\sim}H^1(j_{**!}V)
\]
along the oriented interval $I=[0,1]$, where $\Phi$ denotes
vanishing cycles at $1$.  Both spaces carry mixed Hodge--Tate
structures.  Semi-holonomy respects their weight filtrations but
need not respect their Hodge filtrations.

The Hodge filtration of a mixed Hodge--Tate structure canonically
splits its weight filtration over $\C$.  These splittings, together
with semi-holonomy and its associated graded map, form a comparison
square.  Its failure to commute defines a unipotent automorphism,
the Hodge defect of shrinking along $I$.  Our main calculation expresses
the entries of its matrix by
convergent iterated integrals.  For a word in the two Kummer
classes, it gives an explicit finite matrix.

The resulting universal series is a component of the inverse
Drinfeld associator, multiplied by an explicit local monodromy
factor.  Its coefficients belong to the algebra generated by
multiple zeta values and $2\pi i$.  Comparing the semi-holonomies
at the two ends of the interval recovers the full holonomy.
This presents the associator through the comparison of the two
cohomological functors.

This comparison is related to the approach of Deligne and Terasoma
\cite{DTerICM,DeligneTerasoma}, in which multiplicative convolution and
vanishing cycles lead to the transport algebra and its harmonic
coproduct; see also \cite{EnriquezFurusho}.
The semi-holonomy construction of \cite{MarkarianConvolution}
provides the topological framework for this relation.  The companion
paper \cite{MarkarianDefect} develops its Hodge-theoretic
interpretation: it extends the defect calculated here to more general
mixed Hodge modules and applies it to the regularized double shuffle
relations through multiplicative convolution.

To retain periods as numbers, we equip the mixed Hodge--Tate
structures with rational splittings of their Betti weight
filtrations.  The final section sketches an $A_\infty$
interpretation of the coefficients of the Hodge defect using these
weight-framed extensions and Beilinson--Deligne cohomology
\cite{BeilinsonAbsolute,CarlsonHain}.  In this description the
coefficients record the action on successive extension classes.

Related $A_\infty$ and bar constructions occur in the work of
Arias Abad--Sch\"atz, Terasoma, and Goncharov
\cite{AriasAbadSchaetz,TerasomaBar,GoncharovHidden}.  The comparison
considered here comes from shrinking the $j_{**!}$-extension:
it gives a Hodge-theoretic interpretation of the semi-holonomy
in \cite{MarkarianConvolution}.  The relation with Drinfeld's
construction \cite{Drinfeld} is established after the period
calculation.

The same point of view suggests a motivic refinement, with
compatible realizations and framings in a category of mixed Tate
motives.  Such a refinement is beyond the scope of this paper.

Section~\ref{sec:iterated-kummer} constructs the Hodge realizations
and the comparison square.  Section~\ref{sec:comparison-defect}
computes the Hodge defect and its word coefficients.
Section~\ref{sec:holonomy-associator} discusses holonomy and the
associator.  Section~\ref{sec:ainfty-category} gives the formal
$A_\infty$ interpretation.

\medskip
\noindent\textit{Acknowledgements.}
I would like to thank IHES for hospitality and excellent
        working conditions. 
This project was supported by the PAUSE program and the ITI IRMIA++.

\section{Iterated Kummer extensions}
\label{sec:iterated-kummer}

\subsection{Kummer variations}
\label{subsec:kummer}

Let
\[
X=\PP\setminus\{0,1,\infty\}.
\]
We use the convention in which $\Q(n)$ has weight $-2n$.
Recall that a mixed Hodge structure is Hodge--Tate if its
weight-graded pieces are direct sums of Tate structures.
For a variation, this condition is imposed on every fiber.

The two basic objects are the Kummer variations $K_z$ and $K_{1-z}$:
\begin{equation}
\begin{tikzcd}[cramped, sep=small]
0 \arrow[r]
& \Q_X(1) \arrow[r]
& K_z \arrow[r]
& \Q_X(0) \arrow[r]
& 0
\end{tikzcd}
\label{eq:kummer-z}
\end{equation}
and
\begin{equation}
\begin{tikzcd}[cramped, sep=small]
0 \arrow[r]
& \Q_X(1) \arrow[r]
& K_{1-z} \arrow[r]
& \Q_X(0) \arrow[r]
& 0 .
\end{tikzcd}
\label{eq:kummer-one-minus-z}
\end{equation}
We use the standard Kummer class for $z$ and the negative of the
standard class for $1-z$.  Thus $K_{1-z}$ below has the opposite
Tate framing to the usual Kummer extension of $1-z$.  With this
convention their de Rham extension classes are represented by
\[
\omega_0=\frac{dz}{z},
\qquad
\omega_1=\frac{dz}{1-z}.
\]

\begin{definition}
We write $\mathcal T(X)$
for the category of admissible unipotent
variations of mixed Hodge--Tate structures with constant Tate graded
pieces.  These are the iterated extensions of constant Tate
variations.  A rational local system and its comparison with the
filtered flat bundle are part of the data of an object.
\end{definition}

\subsection{Word connections}
\label{subsec:word-extensions}
We now introduce a family of finite logarithmic connections indexed
by words in the two symbols $0$ and $1$.  These connections will be
used only to record the coefficients of the universal matrices that
appear later.  Let
\[
w=i_1\cdots i_n,\qquad i_r\in\{0,1\},
\]
let $E_w$ be the graded complex vector space with basis
$q_0,\ldots,q_n$, where $q_r$ has Tate degree $r$.  Put
\begin{equation}
W_{-2r}E_w=\langle q_r,\ldots,q_n\rangle,
\qquad
F^pE_w=\langle q_r\mid r\leq-p\rangle.
\label{eq:word-weight-filtration}
\end{equation}
Define $N_0^{(w)},N_1^{(w)}$ by
\begin{equation}
N_a^{(w)}q_r=
\begin{cases}
q_{r+1},&a=i_{n-r},\\
0,&a\ne i_{n-r},
\end{cases}
\quad 0\le r<n,
\qquad N_a^{(w)}q_n=0.
\label{eq:word-residues}
\end{equation}
The word connection $\mathcal E_w$ is the filtered logarithmic
connection on $E_w\otimes\mathcal O_X$ given by
\begin{equation}
\nabla_w=d-N_0^{(w)}\omega_0-N_1^{(w)}\omega_1.
\label{eq:word-connection}
\end{equation}
Its canonical extension to $\PP$ is the trivial bundle.  The empty
word gives the trivial rank-one connection.  For $n>0$ there is an
exact sequence of filtered connections
\begin{equation}
0\longrightarrow\C(n)\otimes\mathcal O_X
\longrightarrow\mathcal E_{i_1\cdots i_n}
\longrightarrow\mathcal E_{i_2\cdots i_n}
\longrightarrow0.
\label{eq:word-extension-sequence}
\end{equation}
Here $\C(n)$ denotes the one-dimensional graded space of Tate degree
$n$; it does not specify a rational Betti realization.  In the
splitting given by the $q_r$, the off-diagonal entry of the
connection is $-\omega_{i_1}$ times the map $q_{n-1}\mapsto q_n$.

Matrices act on columns, so the last letter of $w$ acts first.
Consequently the $(n,0)$-entry of a noncommutative series evaluated
at $N_0^{(w)},N_1^{(w)}$ is its coefficient of $w$.
More generally, the $(b,a)$-entry extracts the coefficient of
$i_{n-b+1}\cdots i_{n-a}$.

When the same formulas
are applied to an actual object of $\mathcal T(X)$, the matrices
$N_0,N_1$ are the connection matrices of its logarithmic de Rham
realization, as in Subsection~\ref{subsec:nearby-cycle-calculation}.

\subsection{Shrinking}
\label{subsec:shrinking}

We use the factorization
\begin{equation}
\begin{tikzcd}[cramped, sep=small]
X \arrow[r, "j_1"]
& \Gm \arrow[r, "j_0"]
& \mathbb A^1 \arrow[r, "j_\infty"]
& \PP ,
\end{tikzcd}
\label{eq:embeddings}
\end{equation}
where $j_a$ adds the point $a$.  We omit the derived-functor notation
and put
\begin{equation}
j_{**!}=j_{\infty!}j_{0*}j_{1*}.
\label{eq:star-star-shriek}
\end{equation}
We use unipotent cycles with the perverse shift:
$\Phi_f=\phi_{f,1}[-1]$ and $\Psi_f=\psi_{f,1}[-1]$.
The shift $L[1]$ of a local system $L$ on $X$ is perverse.
Write $i_1:\{1\}\hookrightarrow\Gm$ and
$\mathcal F=j_{1*}L[1]$.  Since this is a derived $*$-extension,
$i_1^!\mathcal F=0$.  The variation triangle
\[
i_1^!\mathcal F\longrightarrow\Phi_{1-z}(\mathcal F)
\xrightarrow{\,\var\,}\Psi_{1-z}(\mathcal F)
\longrightarrow i_1^!\mathcal F[1]
\]
therefore makes $\var$ an isomorphism.  This does not
require $T_1-1$ to be invertible: in
$\var\circ\can=T_1-1$, the map $\can$ need not be an
isomorphism.  We take the nearby-cycle lift determined by the
positive real ray of $1-z$.

\begin{prop}[Shrinking]
\label{prop:shrinking}
For every unipotent local system $L$ on $X$ one has
\begin{equation}
H^i\!\left(\PP,j_{**!}L[1]\right)=0
\qquad (i\ne0),
\label{eq:cohomological-purity}
\end{equation}
and the oriented interval $I=[0,1]$ induces a natural isomorphism
\begin{equation}
\begin{tikzcd}[cramped, sep=small]
\Phi_{1-z}\!\left(j_{1*}L[1]\right)
  \arrow[r, "\varphi_I", "\sim"']
& H^0\!\left(\PP,j_{**!}L[1]\right).
\end{tikzcd}
\label{eq:shrinking-isomorphism}
\end{equation}
\end{prop}

\begin{proof}
Set $\mathcal F=j_{1*}L[1]$.  The supported-cohomology shrinking
of \cite[Proposition~2]{MarkarianConvolution} gives
\begin{equation}
H^\bullet_{(0,1]}(\Gm,\mathcal F)
  \xrightarrow{\ \sim\ }
H^\bullet\!\left(\PP,j_{\infty!}j_{0*}\mathcal F\right).
\label{eq:shrinking-supported-cohomology}
\end{equation}
The left-hand side is concentrated in degree zero.  Its
degree-zero group is $\Phi_{1-z}(\mathcal F)$.  Composing
\eqref{eq:shrinking-supported-cohomology} with this identification
gives \eqref{eq:shrinking-isomorphism}.  Equivalently, a class is
represented by a horizontal section over the interior of $I$, and
the variation morphism identifies its limiting value in the nearby-cycle
fiber at $1$.
\end{proof}

\begin{rem}
Vanishing \eqref{eq:cohomological-purity} is implied by
the exactness properties of push-forwards along affine morphisms of perverse sheaves, see \cite[Remark 1]{MarkarianConvolution}.
The same holds for permutations of the boundary types $(**!)$
and, by Verdier duality, for types $(!!*)$.  Relations between the
corresponding cohomology groups are discussed in
Subsection~\ref{subsec:duality-path-torsor}.

The space $H^0\!\left(\PP,j_{**!}L[1]\right)$ may be identified
with the horizontal sections of $L$ over $(0,1)$.  Under this
identification, $\var\varphi_I^{-1}$ takes a section
to its nearby-cycle value at $1$; see
\cite[Section~1]{MarkarianConvolution}.
\end{rem}

\subsection{Logarithmic realization}
\label{subsec:logarithmic-realization}

Let $\mathbb V\in\mathcal T(X)$, with filtered flat bundle
$(\mathcal V,\nabla,F,W)$, and put $D=\{0,1,\infty\}$.
Write $\overline{\mathcal V}$ for the Deligne canonical extension
to $\PP$.  Admissibility extends $F$ to subbundles
$F^p\overline{\mathcal V}$.  We shall use the complex
\begin{equation}
\DR_{**!}(\mathbb V)=
\left[
\overline{\mathcal V}(-\infty)
\xrightarrow{\ \nabla\ }
\overline{\mathcal V}\otimes
\Omega^1_{\PP}(\log D)(-\infty)
\right][1]
\label{eq:logarithmic-complex}
\end{equation}
with filtration
\begin{equation}
F^p\DR_{**!}(\mathbb V)=
\left[
F^p\overline{\mathcal V}(-\infty)
\xrightarrow{\ \nabla\ }
F^{p-1}\overline{\mathcal V}\otimes
\Omega^1_{\PP}(\log D)(-\infty)
\right][1].
\label{eq:hodge-filtered-logarithmic-complex}
\end{equation}
Here $\Omega^1_{\PP}(\log D)$ consists of one-forms with at most
logarithmic poles along $D$, and $(-\infty)$ denotes tensor product
with $\mathcal O_{\PP}(-\{\infty\})$.  The displayed terms lie in
degrees $0,1$ before the perverse shift $[1]$.  Griffiths
transversality gives the index $p-1$ in the one-form term.
Put
\begin{equation}
H^1_{\mathrm{dR}}(j_{**!}\mathbb V)
=\mathbb H^0\!\left(\PP,\DR_{**!}(\mathbb V)\right).
\label{eq:de-rham-cohomology}
\end{equation}

\begin{prop}[Logarithmic Hodge realization]
\label{prop:logarithmic-hodge-realization}
The space $H_B^1(j_{**!}V_{\Q})$ carries a functorial mixed Hodge
structure, denoted by $H^1(j_{**!}\mathbb V)$.  Its de Rham
realization and Hodge filtration are computed by
\eqref{eq:logarithmic-complex} and
\eqref{eq:hodge-filtered-logarithmic-complex}.  There is a functorial
comparison quasi-isomorphism
\begin{equation}
j_{**!}V_{\Q}[1]\otimes_{\Q}\C
\xrightarrow{\ \sim\ }
\DR_{**!}(\mathbb V)^{\mathrm{an}}.
\label{eq:logarithmic-comparison}
\end{equation}
\end{prop}

\begin{proof}
We construct the mixed Hodge complex before applying the shift
$[1]$.  Let $j:X\hookrightarrow\PP$.  The ordinary logarithmic
complex
\[
\mathcal K_{\mathrm{dR}}=
\left[
\overline{\mathcal V}\xrightarrow{\nabla}
\overline{\mathcal V}\otimes\Omega^1_{\PP}(\log D)
\right]
\]
is the complex component of a cohomological mixed Hodge complex
$\mathcal K$ representing $Rj_*V_{\Q}$.
Its Hodge filtration has terms $F^p\overline{\mathcal V}$ and
$F^{p-1}\overline{\mathcal V}\otimes\Omega^1(\log D)$; see
\cite[Theorem~(4.1), (4.5)--(4.6), and (4.8)--(4.13)]{SteenbrinkZucker}.
The underlying logarithmic comparison is also given in
\cite[Chapter~II, Corollary~6.10]{DeligneED}.

Write $\mathcal W$ for the cohomological weight filtration on
these complexes.  The constant Tate graded pieces imply that
each local monodromy logarithm sends $W_k$ into $W_{k-2}$.
Hence its relative monodromy filtration is $W$.  In this case
the filtration of \cite[(4.8)--(4.9)]{SteenbrinkZucker} is
\[
\begin{split}
\mathcal W_k\mathcal K_{\mathrm{dR}}^0
 &=W_k\overline{\mathcal V},\\
\mathcal W_k\mathcal K_{\mathrm{dR}}^1
 &=\left\{\alpha\in W_k\overline{\mathcal V}
       \otimes\Omega^1_{\PP}(\log D):
       \operatorname{Res}_a\alpha\in W_{k-1}
       (\overline{\mathcal V}|_a)\text{ for all }a\in D\right\}.
\end{split}
\]

We now impose the boundary condition at $\infty$.  Let
$i:\{\infty\}\hookrightarrow\PP$ and use the coordinate $t=1/z$.
Put $E_\infty=\overline{\mathcal V}|_\infty$ and
$R_\infty=\operatorname{Res}_\infty\nabla$.  The limiting mixed
Hodge structure $\Psi_\infty$ has complex realization $E_\infty$,
with the extended Hodge filtration and weight filtration $W$;
see \cite[(3.13)(ii)--(iii)]{SteenbrinkZucker}.
Let $T_\infty$ be its monodromy and
$N_\infty:\Psi_\infty\to\Psi_\infty(-1)$ its logarithm, with the
Tate twist.  In de Rham coordinates,
\[
T_\infty=\exp(-2\pi iR_\infty),
\qquad (N_\infty)_{\mathrm{dR}}=-R_\infty.
\]
Thus the boundary mixed Hodge complex has realizations
\[
\mathcal B_{\Q}=
[\Psi_{\infty,\Q}\xrightarrow{-N_\infty}
 \Psi_{\infty,\Q}(-1)],
\qquad
\mathcal B_{\mathrm{dR}}=
[E_\infty\xrightarrow{R_\infty}E_\infty].
\]
The Tate twist shifts the Hodge filtration in degree one; the
cohomological weight filtration also accounts for that degree.
Explicitly,
\[
F^p\mathcal B_{\mathrm{dR}}
 =[F^pE_\infty\longrightarrow F^{p-1}E_\infty],
\qquad
\mathcal W_k\mathcal B_{\mathrm{dR}}
 =[W_kE_\infty\longrightarrow W_{k-1}E_\infty].
\]
Since $R_\infty W_k\subset W_{k-2}$, the differential on
$\operatorname{gr}^{\mathcal W}_k\mathcal B$ is zero.  Its terms
in degrees $0,1$ are pure of weights $k,k+1$, respectively.
Therefore $\mathcal B$ is a cohomological mixed Hodge complex.

Evaluation and residue define a morphism
\[
r_{\mathrm{dR}}:\mathcal K_{\mathrm{dR}}\longrightarrow
i_*\mathcal B_{\mathrm{dR}},
\qquad
r^0(s)=s|_\infty,\quad
r^1(\alpha)=\operatorname{Res}_\infty\alpha.
\]
It is a chain map because
$\operatorname{Res}_\infty(\nabla s)=R_\infty(s|_\infty)$,
and the displayed formulas show that it preserves $\mathcal W$
and $F$.  On the rational side it represents the restriction
$Rj_*V_{\Q}\to i_*i^*Rj_*V_{\Q}$.
Indeed, choose a local logarithmic frame in which
$\nabla=d+R_\infty\,dt/t$.  The positive powers of $t$ in the
logarithmic complex contract using $(m+R_\infty)^{-1}$, $m\geq1$;
these operators exist and preserve $W$.  The remaining complex is
$\mathcal B_{\mathrm{dR}}$.  Logarithmic comparison identifies it
with the local monodromy complex; a rational model for $Rj_*V_{\Q}$
is described in \cite[(4.6)]{SteenbrinkZucker}.
In standard Tate generators, the chain
isomorphism from the differential $-\log T_\infty$ to
$T_\infty-1$ is the identity in degree zero and
\[
-\sum_{m\geq0}\frac{(\log T_\infty)^m}{(m+1)!}
\]
in degree one.  The sum is finite and invertible.  This supplies
the compatible rational boundary realization of $r$.

Take the homotopy fiber
\begin{equation}
\mathcal C=\operatorname{Cone}
  (\mathcal K\xrightarrow{r}i_*\mathcal B)[-1].
\label{eq:mixed-boundary-cone}
\end{equation}
Thus $\mathcal C^n=\mathcal K^n\oplus i_*\mathcal B^{n-1}$,
with differential $(a,b)\mapsto(d a,r(a)-d b)$.  Give it the
filtrations
\[
\begin{split}
F^p\mathcal C^n
 &=F^p\mathcal K^n\oplus i_*F^p\mathcal B^{n-1},\\
\mathcal W_k\mathcal C^n
 &=\mathcal W_k\mathcal K^n
   \oplus i_*\mathcal W_{k+1}\mathcal B^{n-1}.
\end{split}
\]
The term involving $r$ vanishes on the weight-graded complex, so
\[
\operatorname{gr}^{\mathcal W}_k\mathcal C
=\operatorname{gr}^{\mathcal W}_k\mathcal K
 \oplus i_*\operatorname{gr}^{\mathcal W}_{k+1}
          \mathcal B[-1].
\]
Both summands are cohomological Hodge complexes of weight $k$.
This proves that $\mathcal C$ is a cohomological mixed Hodge
complex.  The localization triangle at $\infty$ identifies
$\mathcal C_{\Q}$ with $j_{\infty!}j_{0*}j_{1*}V_{\Q}$.
Its hypercohomology consequently gives the asserted mixed Hodge
structure.  The weight convention is
\[
W_mH^1(j_{**!}\mathbb V)=
\operatorname{im}\left(
\mathbb H^1(\PP,\mathcal W_{m-1}\mathcal C)
\longrightarrow\mathbb H^1(\PP,\mathcal C)\right).
\]

Finally, $r_{\mathrm{dR}}$ is degreewise surjective on every
$F^p$: evaluation and residue lift locally through the Hodge
subbundles.  Its kernel is
\[
\left[
\overline{\mathcal V}(-\infty)\xrightarrow{\nabla}
\overline{\mathcal V}\otimes
\Omega^1_{\PP}(\log D)(-\infty)
\right],
\]
and its $F^p$ part has the two terms displayed in
\eqref{eq:hodge-filtered-logarithmic-complex}, before the shift.
The inclusion of this kernel into $\mathcal C_{\mathrm{dR}}$
is therefore an $F$-filtered quasi-isomorphism.  This is the
one-boundary-point version of the collapse used in
\cite[p.~521, proof of Theorem~(4.30)]{SteenbrinkZucker}.
It proves both the claimed Hodge filtration and the comparison
\eqref{eq:logarithmic-comparison} after applying $[1]$.
All constructions are functorial in $\mathbb V$.
\end{proof}

The cone \eqref{eq:mixed-boundary-cone} defines the weights.
The smaller complex \eqref{eq:logarithmic-complex}, with its
$F$-filtered comparison, is the one used to calculate the Hodge
filtration in Section~\ref{sec:comparison-defect}.

At $1$, we give the vanishing cycles the mixed Hodge structure
transported from the limiting mixed Hodge structure
$\Psi_{1-z}(\mathbb V[1])(-1)$ by the logarithmic variation
isomorphism; see \cite[Section~1.1 and Lemma~1.4]{MorihikoSaito}.
We write
\begin{equation}
\Phi(j_{**!}\mathbb V)
=\Phi_{1-z}\!\left(j_{1*}\mathbb V[1]\right)
\label{eq:vanishing-cycle-mhs}
\end{equation}
and denote its Betti and de Rham realizations by
$\Phi_B(j_{**!}\mathbb V)$ and
$\Phi_{\mathrm{dR}}(j_{**!}\mathbb V)$.

\begin{prop}[Hodge--Tate property]
\label{prop:hodge-tate-output}
For $\mathbb V\in\mathcal T(X)$, the mixed Hodge structures
$\Phi(j_{**!}\mathbb V)$ and $H^1(j_{**!}\mathbb V)$ are
Hodge--Tate.  Both constructions commute with Tate twists: for every
$n$ there are natural isomorphisms
\begin{equation}
\Phi(j_{**!}(\mathbb V(n)))
  \simeq\Phi(j_{**!}\mathbb V)(n),
\qquad
H^1(j_{**!}(\mathbb V(n)))
  \simeq H^1(j_{**!}\mathbb V)(n).
\label{eq:tate-twist-compatibility}
\end{equation}
\end{prop}

\begin{proof}
Shrinking and the long exact cohomology sequence show that
$\mathbb V\mapsto H^1(j_{**!}\mathbb V)$ is exact on
$\mathcal T(X)$.  Vanishing cycles are exact after this perverse
shift.  On a constant Tate variation both functors have
value $\Q(n-1)$.  For cohomology this follows from the logarithmic
representative $\vartheta=dz/(z(1-z))$, of type $(1,1)$; for
vanishing cycles it follows from the variation isomorphism
\[
\Var\colon
\Phi_{1-z}(j_{1*}\mathbb V[1])
\xrightarrow{\sim}\Psi_{1-z}(\mathbb V[1])(-1).
\]
Here and below $\Var$ is the logarithmic variation
morphism, compatible with the Hodge realizations.  Its underlying
Betti map differs from the topological variation morphism by the
invertible power series relating $\log T$ to $T-1$.
Exactness and induction on the Tate filtration prove the assertion.
Tate twists commute with both constructions.
\end{proof}

\subsection{Comparison square}
\label{subsec:comparison-square}

From now on we suppress the common perverse shift and write
\[
H^1(j_{**!}V):=H^0(\PP,j_{**!}V[1]).
\]
We denote the Betti realization of the resulting cohomology by
\[
H_B^1(j_{**!}\mathbb V)
=H^0\!\left(\PP,j_{**!}V_{\Q}[1]\right).
\]
The shrinking isomorphism is compatible with weights.

\begin{prop}[Weight compatibility]
\label{prop:weight-compatibility}
For every $\mathbb V\in\mathcal T(X)$, the isomorphism
\begin{equation}
\varphi_I\colon
\Phi_B(j_{**!}\mathbb V)
\xrightarrow{\ \sim\ }
H_B^1(j_{**!}\mathbb V)
\label{eq:weighted-shrinking}
\end{equation}
is strictly compatible with the weight filtration.
\end{prop}

\begin{proof}
Write $H(\mathbb V)=H^1(j_{**!}\mathbb V)$ in this proof.
Exactness and the calculation on constant Tate variations give
\[
W_mH(\mathbb V)=H(W_{m-2}\mathbb V),\qquad
W_m\Phi(j_{**!}\mathbb V)=\Phi(j_{**!}W_{m-2}\mathbb V).
\]
Indeed, the associated graded objects of these filtrations are pure
of the required weights, so these are the weight filtrations by
uniqueness.  Naturality of Betti shrinking with respect to
$W_k\mathbb V\hookrightarrow\mathbb V$ now proves strict weight
compatibility.  No filtered de Rham shrinking map is used.
\end{proof}

We now use the Hodge--Tate property.  By Proposition
\ref{prop:hodge-tate-output}, the mixed Hodge structures
$\Phi(j_{**!}\mathbb V)$ and $H^1(j_{**!}\mathbb V)$ are
Hodge--Tate.  Put
\[
\Phi_{B,\C}:=\Phi_B(j_{**!}\mathbb V)\otimes_\Q\C,
\qquad
H^1_{B,\C}:=H_B^1(j_{**!}\mathbb V)\otimes_\Q\C,
\]
and let $\varphi_{I,\C}=\varphi_I\otimes_\Q\C$.

For a Hodge--Tate mixed Hodge structure $H$, the Hodge filtration
gives a canonical splitting of the weight filtration over $\C$:
\begin{equation}
H_{B,\C}=\bigoplus_p
\bigl(F^pH_{B,\C}\cap W_{2p}H_{B,\C}\bigr),
\qquad
F^pH_{B,\C}\cap W_{2p}H_{B,\C}
\xrightarrow{\ \sim\ }
\operatorname{gr}_{2p}^W H_{B,\C}.
\label{eq:hodge-tate-splitting}
\end{equation}
Denote the inverse of the second map, summed over $p$, by
\begin{equation}
s_{H,B}^F\colon\operatorname{gr}^W H_{B,\C}
\xrightarrow{\ \sim\ }H_{B,\C}.
\label{eq:hodge-splitting}
\end{equation}
Applying this construction to the two Hodge--Tate structures gives
splittings $s_{\Phi,B}^F$ and $s_{H^1,B}^F$.  Since $\varphi_I$ is
strictly compatible with weights, it induces an isomorphism
\[
\operatorname{gr}^W(\varphi_{I,\C})\colon
\operatorname{gr}^W\Phi_{B,\C}
\xrightarrow{\ \sim\ }
\operatorname{gr}^W H^1_{B,\C}.
\]
The comparison square is
\begin{equation}
\begin{tikzcd}[column sep=large, row sep=large]
\operatorname{gr}^W\Phi_{B,\C}
  \arrow[r, "s_{\Phi,B}^F", "\sim"']
  \arrow[d, "\operatorname{gr}^W(\varphi_{I,\C})"', "\sim"]
& \Phi_{B,\C}
  \arrow[d, "\varphi_{I,\C}", "\sim"'] \\
\operatorname{gr}^W H^1_{B,\C}
  \arrow[r, "s_{H^1,B}^F"', "\sim"]
& H^1_{B,\C}.
\end{tikzcd}
\label{eq:comparison-square}
\end{equation}
The horizontal maps are generally only $\C$-linear; they need not be
rational.  The two routes through \eqref{eq:comparison-square} need
not agree.  Their discrepancy is measured by the following automorphism,
called the \emph{Hodge defect of shrinking along $I$}:
\begin{equation}
\mathcal G_I(\mathbb V)=
 s_{\Phi,B}^F\circ
 \bigl(\operatorname{gr}^W(\varphi_{I,\C})\bigr)^{-1}\circ
 \bigl(s_{H^1,B}^F\bigr)^{-1}\circ
 \varphi_{I,\C}
\quad\text{of}\quad\Phi_{B,\C}.
\label{eq:defect-automorphism}
\end{equation}
The Hodge defect is the identity if and only if $\varphi_{I,\C}$
respects the Hodge filtrations.  It preserves $W$ and induces the
identity on $\operatorname{gr}^W$;
in particular, it is unipotent.  Both objects have a common shift of
Tate degree: on $\Q_X(n)$ their value is $\Q(n-1)$.  We keep the
usual cohomological notation in Sections~1--3 and tensor both
functors with the Tate object $\Q(1)$ in Section~4.  This common twist leaves the matrix
of the Hodge defect unchanged.

For later use, let
\begin{align}
c_\Phi\colon
\Phi_{B,\C}&\xrightarrow{\ \sim\ }
\Phi_{\mathrm{dR}}(j_{**!}\mathbb V),
\label{eq:comparison-phi}\\
c_H\colon
H^1_{B,\C}&\xrightarrow{\ \sim\ }
H^1_{\mathrm{dR}}(j_{**!}\mathbb V)
\label{eq:comparison-h}
\end{align}
be the Betti--de Rham comparison isomorphisms.  We transport the
Hodge splittings to de Rham realizations by
\begin{equation}
s_{H,\mathrm{dR}}^F
=c_H\circ s_{H,B}^F\circ(\operatorname{gr}^W c_H)^{-1},
\qquad H=\Phi,H^1.
\label{eq:hodge-splitting-dr}
\end{equation}

\section{The Hodge defect}
\label{sec:comparison-defect}

\subsection{Logarithmic primitives and shrinking}
\label{subsec:nearby-cycle-calculation}

Let $\mathbb V\in\mathcal T(X)$.  Admissibility implies that
$\operatorname{gr}_F\operatorname{gr}_W\overline{\mathcal V}$
is locally free.  At each boundary point the monodromy logarithm
lowers $W$ by two, so its relative monodromy filtration is $W$.
The limiting weight-graded pieces are the same constant Tate
structures.  Thus $F$ and $W$ remain opposed on the boundary
fibers, and projection gives bundle isomorphisms
\[
\mathcal S_r:=F^{-r}\overline{\mathcal V}
             \cap W_{-2r}\overline{\mathcal V}
\xrightarrow{\sim}\operatorname{gr}_{-2r}^W\overline{\mathcal V}.
\]
Each $\mathcal S_r$ is therefore the corresponding trivial
weight-graded bundle.  With $D=\{0,1,\infty\}$, compatibility
with $W$ and Griffiths transversality give
\[
\nabla\mathcal S_r\subset
(\mathcal S_r\oplus\mathcal S_{r+1})
\otimes\Omega^1_{\PP}(\log D).
\]
The $\mathcal S_r$ component is the constant graded connection.
Taking constant sections of these bundles consequently gives
\[
\overline{\mathcal V}=E\otimes\mathcal O_{\PP},
\qquad E=\bigoplus_r E_r,
\qquad
\nabla=d-N_0\omega_0-N_1\omega_1.
\]
Here $E_r=H^0(\PP,\mathcal S_r)$, and
$N_0,N_1:E_r\to E_{r+1}$.  The residues of $\nabla$ at $0$ and $1$
are $-N_0$ and $N_1$, respectively.

We first describe the two Hodge frames used in the comparison
square.  The logarithmic complex \eqref{eq:logarithmic-complex}
gives
\[
H^1_{\mathrm{dR}}(j_{**!}\mathbb V)
 =\{[v\vartheta]:v\in E\},
\qquad \vartheta=\frac{dz}{z(1-z)}.
\]
This follows because $\mathcal O_{\PP}(-\infty)$ has no cohomology
and $\Omega^1_{\PP}(\log D)(-\infty)$ is generated by $\vartheta$.
Moreover,
$F^pH^1_{\mathrm{dR}}=\bigoplus_{r\leq1-p}E_r\vartheta$.
Thus the oriented residue is an isomorphism of filtered complex
vector spaces
\begin{equation}
\rho_1=-\operatorname{Res}_{z=1}\colon
H^1_{\mathrm{dR}}(j_{**!}\mathbb V)\xrightarrow{\sim}E(-1),
\qquad [v\vartheta]\longmapsto v.
\label{eq:oriented-residue}
\end{equation}
At $1$ we use the positive real ray in the coordinate $1-z$.
The logarithmic nearby-cycle realization is $E$, with its Tate
grading, and logarithmic variation gives
\[
\Var\colon
\Phi_{\mathrm{dR}}(j_{**!}\mathbb V)\xrightarrow{\sim}E(-1).
\]
Both maps preserve the Hodge and weight filtrations.  After
identifying their associated graded spaces with $E(-1)$, the Hodge
splittings are therefore
\begin{equation}
s_{H^1,\mathrm{dR}}^F=\rho_1^{-1},
\qquad s_{\Phi,\mathrm{dR}}^F=\Var^{-1}.
\label{eq:hodge-identification-residue}
\end{equation}
In these coordinates the associated graded shrinking is the
identity, as the calculation on constant Tate variations shows.
We use this frame for the matrix of the Hodge defect and suppress
the common Tate twist in matrices.

The calculation consists in integrating $v\vartheta$ with the
boundary condition at $\infty$ and reading its Betti class from
the jump across $I$.

\begin{prop}[Logarithmic primitive]
\label{prop:logarithmic-primitive}
For every $v\in E$ there is a unique single-valued holomorphic
function $S_v:\PP\setminus[0,1]\to E$ satisfying
\begin{equation}
\nabla S_v=v\vartheta,
\qquad S_v(z)=O(z^{-1})\quad(z\longrightarrow\infty).
\label{eq:inhomogeneous-primitive}
\end{equation}
Let $S_{v,+}$ and $S_{v,-}$ be its boundary values from the upper
and lower half-planes, respectively, and put
\begin{equation}
J_v(t)=\frac{S_{v,-}(t)-S_{v,+}(t)}{2\pi i},
\qquad 0<t<1.
\label{eq:jump-representative}
\end{equation}
Then $J_v$ is horizontal and has polynomial logarithmic growth at
$0$ and $1$.  One has
\begin{equation}
S_v(z)=\int_0^1\frac{J_v(t)}{z-t}\,dt,
\qquad
(1+N_0-N_1)\int_0^1J_v(t)\,dt=v.
\label{eq:cauchy-primitive}
\end{equation}
Under Betti--de Rham comparison, the horizontal section
$2\pi iJ_v$ represents
$\var\,\varphi_I^{-1}c_H^{-1}[v\vartheta]$.
\end{prop}

\begin{proof}
In the coordinate $u=z^{-1}$, the equation has a unique power
series solution with zero constant term: the recursion only
inverts operators of the form $m+N_0-N_1$, $m\geq1$.  These are
invertible because $N_0-N_1$ is nilpotent.  The regular-singular
equation gives a convergent solution, which continues to the
simply connected slit sphere.  Successive integration along the
Tate filtration gives polynomial logarithmic growth at the ends
of the slit.

The difference of the boundary values is horizontal.  Cauchy's
formula, applied around the slit, gives the first equality in
\eqref{eq:cauchy-primitive}.  The small circles around $0$ and $1$
contribute zero in the limit, since their lengths tend to zero
faster than any power of the logarithm grows.  For the constant
connection,
$S_v=v(\log z-\log(z-1))$ has lower-minus-upper jump $2\pi i v$,
which fixes the sign.  Writing
$M_v=\int_0^1J_v(t)\,dt$, we obtain
$S_v(z)=z^{-1}M_v+O(z^{-2})$.  Substitution into
\[
\nabla S_v=v(-z^{-2}+O(z^{-3}))\,dz,
\qquad
\frac{N_0}{z}+\frac{N_1}{1-z}
 =\frac{N_0-N_1}{z}+O(z^{-2}),
\]
gives $(1+N_0-N_1)M_v=v$.

It remains to identify the Betti class.  Put
$U=\PP\setminus I$.  Since $\nabla S_v=v\vartheta$ on $U$, the pair
$(v\vartheta,S_v)$ represents a class in the relative de Rham
complex with support on $I$.  The decay of $S_v$ is precisely the
degree-zero $!$ condition at $\infty$.  Forgetting supports gives
$[v\vartheta]$.  On a small disk about an interior point of $I$,
choose a holomorphic primitive $R$ of $v\vartheta$.  Subtracting
its relative coboundary leaves the two horizontal sections
$S_{v,+}-R$ and $S_{v,-}-R$.  Their difference is
$S_{v,-}-S_{v,+}$.

The $*$-extensions at $0$ and $1$ have zero costalks, so removing
these points from the support does not change the supported
cohomology.  The construction of Proposition~\ref{prop:shrinking}
identifies restriction to the interior with topological
variation at $1$.  With the orientation just
fixed, the difference above therefore represents
$\var\,\varphi_I^{-1}c_H^{-1}[v\vartheta]$.
\end{proof}

To pass from this Betti class to its Hodge coordinate, we must
distinguish topological and logarithmic variation.  In the
nearby-cycle de Rham frame, let
\begin{equation}
T_1=\exp(-2\pi iN_1),
\qquad
U_1=\sum_{m\geq0}\frac{(-2\pi iN_1)^m}{(m+1)!}.
\label{eq:variation-unit}
\end{equation}
This finite sum is invertible, with constant term $1$; the
expression $U_1=(T_1-1)/\log T_1$ denotes this power series, not
division by the nilpotent operator $\log T_1$.
Let $c_\Psi$ denote comparison from Betti nearby cycles to $E$.
Then, in the standard Tate frames,
\begin{equation}
\Var\circ c_\Phi
 =\frac{1}{2\pi i}\,U_1^{-1}c_\Psi \var.
\label{eq:variation-comparison}
\end{equation}
Logarithmic variation is obtained by applying the invertible
power series $\log T_1/(T_1-1)$ to topological variation, with
target $\Psi(-1)$.
The factor $2\pi i$ in \eqref{eq:variation-comparison} accounts for
this Tate twist.  Throughout, $\varphi_I$ remains the shrinking
map of Proposition~\ref{prop:shrinking}.

\subsection{The defect matrix and its word coefficients}
\label{subsec:defect-matrix}

Let $G_1$ be the horizontal matrix solution characterized by
\begin{equation}
dG_1=(N_0\omega_0+N_1\omega_1)G_1,
\qquad
G_1(z)(1-z)^{N_1}\longrightarrow1\quad(z\longrightarrow1).
\label{eq:nearby-solution}
\end{equation}
We take real logarithms on $0<z<1$, using the same nearby-cycle
lift as in Subsection~\ref{subsec:nearby-cycle-calculation}.
Thus a horizontal section with nearby-cycle coordinate $w$ at
$1$ is $G_1(z)w$.

\begin{prop}[Matrix of the Hodge defect]
\label{prop:defect-matrix}
For the shrinking map of Proposition~\ref{prop:shrinking}, the
matrix of the Hodge defect in the nearby-cycle Hodge frame is
\begin{equation}
\mathcal G_I(\mathbb V)=
(1+N_0-N_1)\left(\int_0^1G_1(t)\,dt\right)U_1.
\label{eq:defect-matrix}
\end{equation}
The integral converges entrywise.  This matrix preserves weights
and induces the identity on the associated graded.
\end{prop}

\begin{proof}
Fix $v\in E$ and write $J_v(t)=G_1(t)w_v$.  By
Proposition~\ref{prop:logarithmic-primitive}, the topological
variation of the inverse shrinking of $[v\vartheta]$ has
nearby-cycle coordinate $2\pi i w_v$.  Equation
\eqref{eq:variation-comparison} therefore gives
\[
\Var (c_\Phi\varphi_I^{-1}c_H^{-1}[v\vartheta])
 =U_1^{-1}w_v.
\]
The left-hand side is $\mathcal G_I(\mathbb V)^{-1}v$:
this is the comparison square in the Hodge frames
\eqref{eq:hodge-identification-residue}.  Consequently,
\begin{equation}
J_v(t)=G_1(t)U_1\mathcal G_I(\mathbb V)^{-1}v.
\label{eq:jump-defect}
\end{equation}
Substituting this identity into the second equality of
\eqref{eq:cauchy-primitive} proves \eqref{eq:defect-matrix}.
Polynomial logarithmic growth gives convergence of the integral.
All nonconstant terms in the displayed matrix strictly raise
Tate degree, proving the last assertion.
\end{proof}

The factor $U_1$ is already visible for a single Kummer extension.
For $N_0=0$ and $N_1^2=0$ one has
$G_1(t)=1-N_1\log(1-t)$, and hence
\[
(1-N_1)\int_0^1G_1(t)\,dt=1,
\qquad
\mathcal G_I=1-\pi iN_1.
\]
Thus the integral factor alone is not the matrix of the Hodge defect
of the original shrinking map.

We now give a formula for every entry.  For a word
$w=i_1\cdots i_n$, use the basis $q_0,\ldots,q_n$ of
Subsection~\ref{subsec:word-extensions}.  The $(b,a)$-entry, $b>a$,
reads the consecutive letters
\begin{equation}
v_{ba}=i_{n-b+1}\cdots i_{n-a}.
\label{eq:word-subwords}
\end{equation}
In the integrals below,
$\omega_0(t)=dt/t$ and $\omega_1(t)=dt/(1-t)$.
For a word connection, evaluate \eqref{eq:defect-matrix} at
$N_0^{(w)}$ and $N_1^{(w)}$ to obtain $\mathcal G_I(\mathcal E_w)$.

\begin{prop}[Word coefficients]
\label{prop:word-defect-matrix}
Put $\ell_\varnothing=1$ and $\ell_{1^m}=1$ for $m\geq1$.
Every word containing a zero has a unique expression
$v=a_1\cdots a_r1^m$ with $a_r=0$.  For such a word put
\begin{equation}
\ell_v=\frac{(-1)^r}{m!}
\int_{0<t_1<\cdots<t_r<1}
t_1\,\omega_{a_1}(t_1)\cdots\omega_{a_r}(t_r)
\bigl(-\log(1-t_r)\bigr)^m.
\label{eq:word-moment-integral}
\end{equation}
Define $k_\varnothing=1$ and
\[
k_{au}=\ell_{au}+\varepsilon_a\ell_u,
\qquad \varepsilon_0=1,\quad\varepsilon_1=-1.
\]
Then set
\begin{equation}
g_v=\sum_{v=u1^m}k_u\,\frac{(-2\pi i)^m}{(m+1)!}.
\label{eq:word-coefficient-formula}
\end{equation}
The sum runs over all ways to remove a final string of ones from
$v$, including the empty string ($m=0$).  In particular,
$g_\varnothing=1$.  The matrix of the Hodge defect has entries
\begin{equation}
(\mathcal G_I(\mathcal E_w))_{ba}=
\begin{cases}
0,&b<a,\\
1,&b=a,\\
g_{v_{ba}},&b>a.
\end{cases}
\label{eq:word-defect-entries}
\end{equation}
All the integrals in \eqref{eq:word-moment-integral} converge.
\end{prop}

\begin{proof}
The coefficient of $1^m$ in $G_1(z)$ is
\begin{equation}
f_{1^m}(z)=\frac{(-\log(1-z))^m}{m!}.
\label{eq:pure-one-coefficient}
\end{equation}
For $v=a_1\cdots a_r1^m$, $a_r=0$, successive integration of
\eqref{eq:nearby-solution} from $1$ gives its coefficient as
\begin{equation}
f_v(z)=\frac{(-1)^r}{m!}
\int_{z<t_1<\cdots<t_r<1}
\omega_{a_1}(t_1)\cdots\omega_{a_r}(t_r)
\bigl(-\log(1-t_r)\bigr)^m.
\label{eq:nearby-word-integral}
\end{equation}
The last form is $\omega_0$, so the innermost integral is
$O((1-t)(1+|\log(1-t)|)^m)$ at $1$.  Successive integrations
against either $\omega_0$ or $\omega_1$ preserve a bound of this
form, and give at most polynomial logarithmic growth at $0$.
Therefore $\int_0^1f_v(z)\,dz$
converges absolutely.  Integrating first over $0<z<t_1$ proves
\eqref{eq:word-moment-integral}, while
$\int_0^1f_{1^m}(z)\,dz=1$.

Thus $\ell_v$ is the coefficient of $v$ in $\int_0^1G_1(t)\,dt$.
Left multiplication by $1+N_0-N_1$ gives $k_v$.  Right
multiplication by $U_1$ appends a string of ones, giving exactly
\eqref{eq:word-coefficient-formula}.  The column-action convention
then gives \eqref{eq:word-defect-entries}.
\end{proof}

For example,
\[
\ell_0=-1,\quad \ell_1=1,\quad
\ell_{10}=\zeta(2)-1,\quad \ell_{01}=-1,
\]
so
\[
g_0=0,\qquad g_1=-\pi i,\qquad
g_{10}=\zeta(2),\qquad g_{01}=0.
\]
For the words $10$ and $01$, this gives
\[
\mathcal G_I(\mathcal E_{10})=
\begin{pmatrix}1&0&0\\0&1&0\\\zeta(2)&-\pi i&1\end{pmatrix},
\qquad
\mathcal G_I(\mathcal E_{01})=
\begin{pmatrix}1&0&0\\-\pi i&1&0\\0&0&1\end{pmatrix}.
\]
For $01$, the coefficient $g_{01}$ vanishes, but the entry
corresponding to its final letter $1$ is $-\pi i$.

\section{Holonomy and the Drinfeld associator}
\label{sec:holonomy-associator}

\subsection{Holonomy}
\label{subsec:two-shrinkings}

The shrinking construction applies at either end of the interval.
For a unipotent local system $L$, write
\[
\varphi_{I,a}\colon
\Phi_a(j_{**!}L[1])\xrightarrow{\sim}
H^0(\PP,j_{**!}L[1]),\qquad a=0,1,
\]
where the local parameters are $z$ and $1-z$, respectively, and the
nearby-cycle fibers are taken in the directions determined by $I$.
These maps are the semi-holonomies of
\cite[Section~6]{MarkarianConvolution}.
Since the extension is $*$ at both ends, topological variation gives
isomorphisms
$\var_a\colon\Phi_a(j_{**!}L[1])\xrightarrow{\sim}\Psi_a(L[1])$.

\begin{prop}[Holonomy]
\label{prop:interval-holonomy}
The composition
\begin{equation}
\var_1\circ\varphi_{I,1}^{-1}\circ
\varphi_{I,0}\circ(\var_0)^{-1}
\colon\Psi_0(L[1])\xrightarrow{\sim}\Psi_1(L[1])
\label{eq:interval-holonomy}
\end{equation}
is parallel transport along $I$ from $0$ to $1$.
\end{prop}

\begin{proof}
By the supported-cohomology construction of shrinking, the composition
identifies the nearby-cycle values at the two ends of the same horizontal
section on $(0,1)$; see \cite[Section~1]{MarkarianConvolution}.
This is parallel transport along $I$.
\end{proof}

\subsection{Associator}
\label{subsec:drinfeld-series}

Use noncommuting variables $e_0,e_1$ in place of $N_0,N_1$.
Consider
\begin{equation}
dG=(e_0\omega_0+e_1\omega_1)G.
\label{eq:universal-kz-equation}
\end{equation}
Its canonical solutions are characterized by
\begin{equation}
G_0(z)z^{-e_0}\longrightarrow1\quad(z\longrightarrow0),
\qquad
G_1(z)(1-z)^{e_1}\longrightarrow1\quad(z\longrightarrow1).
\label{eq:canonical-solutions}
\end{equation}
The constant series
\begin{equation}
\Ass=G_1^{-1}G_0
\label{eq:drinfeld-associator}
\end{equation}
is the Drinfeld associator defined in
\cite[Section~2]{Drinfeld}, in the convention of
\eqref{eq:universal-kz-equation}; see also
\cite[Section~4]{LeMurakami}.
The matrix of interval holonomy in these nearby-cycle frames is
$\Ass(N_0,N_1)$.

Write the leftmost-letter decompositions
\begin{equation}
\Ass=1+e_0a_0+e_1a_1,\qquad
\Ass^{-1}=1+e_0b_0+e_1b_1.
\label{eq:associator-left-decomposition}
\end{equation}

\begin{theorem}
\label{thm:defect-associator}
The universal series whose word evaluations are the matrices of the
Hodge defect of Section~\ref{sec:comparison-defect} is
\begin{equation}
\mathcal G_I=(1+e_0a_0)\Ass^{-1}U(e_1)
=(1+e_1b_1)U(e_1),
\label{eq:defect-associator}
\end{equation}
where
\[
U(e_1)=\sum_{m\geq0}\frac{(-2\pi i)^m e_1^m}{(m+1)!}.
\]
Its coefficients belong to the algebra generated by multiple
zeta values and $2\pi i$.  The coefficients of
$\mathcal G_IU(e_1)^{-1}$ are precisely the constant term and the
coefficients of $\Ass^{-1}$ on words beginning in $e_1$.
\end{theorem}

\begin{proof}
All integrals below are understood coefficientwise.
Cancel the common invertible right factor $U(e_1)$.  Since
$G_1=G_0\Ass^{-1}$, Proposition~\ref{prop:defect-matrix} reduces
the first equality to
\begin{equation}
(1+e_0-e_1)\int_0^1G_0(t)\,dt=1+e_0a_0.
\label{eq:associator-moment}
\end{equation}
Let $g_v^0(t)$ be the coefficient of a word $v$ in $G_0(t)$.
For a word $0u$, integration by parts and
$(g_{0u}^0)'=g_u^0/t$ give
\[
\int_0^1g_{0u}^0(t)dt+\int_0^1g_u^0(t)dt
=g_{0u}^0(1).
\]
The limiting value at $1$ exists because the first letter is $0$; it is
the coefficient of $0u$ in $\Ass$.  At $0$, $t g_{0u}^0(t)$ tends to
zero.  For a word $1u$, use
$(g_{1u}^0)'=g_u^0/(1-t)$ and integrate
$((t-1)g_{1u}^0)'$.  Both boundary terms vanish, and hence
\[
\int_0^1g_{1u}^0(t)dt-\int_0^1g_u^0(t)dt=0.
\]
These two identities, together with the constant coefficient,
prove \eqref{eq:associator-moment}.
The equation $\Ass\Ass^{-1}=1$, compared on words beginning in $e_1$,
gives $b_1=-a_1\Ass^{-1}$.  It follows that
$(1+e_0a_0)\Ass^{-1}=1+e_1b_1$, proving the second equality.
The coefficients of $\Ass^{-1}$ belong to the algebra of multiple
zeta values, since those of $\Ass$ do and inversion is polynomial
in each word length.  Multiplication by $U(e_1)$ adjoins only
powers of $2\pi i$, proving the coefficient assertion.
\end{proof}

To compare with the notation of \cite[Section~6]{MarkarianConvolution},
write the rightmost-letter decomposition of any series as

\[
f=1+f_0e_0+f_1e_1,
\qquad f_Y:=1+f_1e_1.
\]
Thus ${\Ass}_Y$ means this operation applied to the associator.
Let $\operatorname{rev}$ reverse every word.
Since $\Ass$ is group-like, its antipode identity is
$\operatorname{rev}(\Ass^{-1})=\Ass(-e_0,-e_1)$.  Thus
\begin{equation}
\mathcal G_I=
\operatorname{rev}\bigl(\Ass(-e_0,-e_1)_Y\bigr)\,U(e_1).
\label{eq:defect-f-y}
\end{equation}
The reversal and signs record the column-action convention and
the direction of the comparison square.  The additional factor
$U(e_1)$ records the difference between topological and
logarithmic variation.

The operation $f\mapsto f_Y$ is the transport-algebra projection
appearing in the description of semi-holonomy in
\cite[Section~6]{MarkarianConvolution}.  The companion paper
\cite[Subsection~2.3]{MarkarianDefect} develops this interpretation
of \eqref{eq:defect-f-y} and extends it to general gluing data at
$1$.  Its convention uses $dz/(z-1)$ in place of $dz/(1-z)$,
so the corresponding letter $e_1$ has the opposite sign.

\subsection{Fundamental groupoid via shrinking and Verdier duality}
\label{subsec:duality-path-torsor}

Let $V$ be a unipotent local system on
$X=\PP\setminus\{0,1,\infty\}$.  We order the boundary symbols as
$(0,1,\infty)$.  All identifications in this subsection are on
the underlying Betti vector spaces.  Here $j_{x,*}$ and $j_{x,!}$ denote the two derived
direct images across the puncture $x$.  We write
\[
j_{\varepsilon_0\varepsilon_1\varepsilon_\infty}V
 :=j_{\infty,\varepsilon_\infty}
   j_{0,\varepsilon_0}j_{1,\varepsilon_1}V,
\qquad \varepsilon_x\in\{!,*\}.
\]

Write $\overrightarrow{ab}$ for the tangential basepoint at $a$
pointing along the real interval towards $b$; the interval is the
component of $\PP(\mathbb R)\setminus\{0,1,\infty\}$ with these
endpoints.  The shrinking construction gives the following
vanishing-cycle identifications:
\begin{align*}
H^1(\PP,j_{**!}V)&\simeq
 \Phi_{\overrightarrow{01}}(j_{**!}V[1])
 \simeq\Phi_{\overrightarrow{10}}(j_{**!}V[1]),\\
H^1(\PP,j_{*!*}V)&\simeq
 \Phi_{\overrightarrow{\infty 0}}(j_{*!*}V[1])
 \simeq\Phi_{\overrightarrow{0\infty}}(j_{*!*}V[1]),\\
H^1(\PP,j_{!**}V)&\simeq
 \Phi_{\overrightarrow{1\infty}}(j_{!**}V[1])
 \simeq\Phi_{\overrightarrow{\infty 1}}(j_{!**}V[1]).
\end{align*}

The natural morphisms $j_!\to j_*$ at one boundary point give the
maps
\begin{align*}
H^1(\PP,j_{!*!}V)&\longrightarrow H^1(\PP,j_{**!}V)\longleftarrow
H^1(\PP,j_{*!!}V),\\
H^1(\PP,j_{!!*}V)&\longrightarrow H^1(\PP,j_{*!*}V)\longleftarrow
H^1(\PP,j_{*!!}V),\\
H^1(\PP,j_{!!*}V)&\longrightarrow H^1(\PP,j_{!**}V)\longleftarrow
H^1(\PP,j_{!*!}V).
\end{align*}

Verdier duality exchanges $!$ and $*$ at every boundary point and gives isomorphisms:
\begin{align*}
H^1(\PP,j_{**!}V)^\vee&\simeq H^1(\PP,j_{!!*}V^\vee),\\
H^1(\PP,j_{*!*}V)^\vee&\simeq H^1(\PP,j_{!*!}V^\vee),\\
H^1(\PP,j_{!**}V)^\vee&\simeq H^1(\PP,j_{*!!}V^\vee).
\end{align*}

For each tangential basepoint $\varepsilon$, nearby cycles give
a canonical identification
\[
\Psi_\varepsilon(V^\vee[1])
\simeq\Psi_\varepsilon(V[1])^\vee.
\]
Here nearby cycles of a local system mean those of its restriction
to a punctured neighborhood.  For $\mathcal F=j_{abc}V[1]$, the map
$\var:\Phi_\varepsilon(\mathcal F)\to\Psi_\varepsilon(\mathcal F)$
is an isomorphism at a $*$-point, whereas
$\can:\Psi_\varepsilon(\mathcal F)\to\Phi_\varepsilon(\mathcal F)$
is an isomorphism at a $!$-point.  In Hodge realizations the global
duality formulas above require a Tate twist $(1)$ on the right,
and at a $*$-point logarithmic variation gives
$\Phi\simeq\Psi(-1)$, as in Section~\ref{sec:iterated-kummer}.

By the Milnor triangles, the natural maps above recover the local
monodromy operators; see
\cite[Section~1.2, Remark~3]{MarkarianConvolution}.
Shrinking recovers parallel transport along each of the three real
intervals.  To describe the fundamental groupoid with all six
tangential basepoints, also fix at each puncture a homotopy class
of a half-circle in the punctured tangent line joining its two
real tangent directions.  The interval transports, these local
transports, and the local monodromies then describe its
parallel-transport representation.

These descriptions suggest a construction of the mixed Hodge
structure on the pro-unipotent fundamental groupoid using
shrinking and duality.  Constructing it and comparing it with
\cite{HainZucker} requires compatibility with tensor products and
the Hodge realizations; we do not develop that construction here.

\section{\texorpdfstring{$A_\infty$}{A-infinity} category}
\label{sec:ainfty-category}

\subsection{Beilinson--Deligne cohomology and weight-framed extensions}
\label{subsec:beilinson-deligne}

Let $Y$ be a smooth complex algebraic variety.  Absolute Hodge
cohomology is defined by
\begin{equation}
R\Gamma_{\mathcal H}(Y,\Q(n))=
R\operatorname{Hom}_{\mathrm{MHS}}
\bigl(\Q(0),R\Gamma(Y,\Q)(n)\bigr),
\qquad
H^r_{\mathcal H}=H^rR\Gamma_{\mathcal H}.
\label{eq:absolute-hodge-ext}
\end{equation}
Here $R\Gamma(Y,\Q)$ is taken in the derived category of mixed
Hodge structures, using Beilinson's mixed Hodge complex construction
\cite[0.1 and Section~3]{BeilinsonAbsolute}.
Thus the weight filtration is retained throughout the construction.

For $X=\PP\setminus\{0,1,\infty\}$ and $n>0$, a familiar
Deligne--Beilinson representative is
\begin{equation}
R\Gamma_{\mathcal H}(X,\Q(n))\simeq
\operatorname{Cone}\left(
\begin{aligned}
R\Gamma_B(X,\Q(n))\oplus F^nR\Gamma_{\mathrm{dR}}(X)\\
{}\longrightarrow R\Gamma_B(X,\C)
\end{aligned}\right)[-1].
\label{eq:beilinson-deligne-complex}
\end{equation}
The arrow is the difference of the two comparison maps.  The de
Rham term is computed by logarithmic forms on $\PP$.
This expression is valid here because
$H^0(X,\Q(n))=\Q(n)$ and
$H^1(X,\Q(n))=\Q(n-1)^{\oplus2}$ have nonpositive weights.
The weight truncation in the absolute Hodge complex is therefore
a quasi-isomorphism.  We do not use the weight-forgetting cone
\eqref{eq:beilinson-deligne-complex} for arbitrary $Y$ or for $n=0$.
Composition is given by the cup product of mixed Hodge complexes;
see \cite[Lemma~1.11 and Remark~1.12]{BeilinsonAbsolute}.

For a point and $b>a$,
\begin{equation}
\operatorname{Ext}^1_{\mathrm{MHS}}(\Q(a),\Q(b))
=\C/(2\pi i)^{b-a}\Q.
\label{eq:tate-extension-quotient}
\end{equation}
The quotient expresses the freedom to change the Betti splitting.
We retain that splitting as part of the data.

\begin{definition}
A \emph{weight framing} of a mixed Hodge--Tate structure $H$ is a
rational splitting
\begin{equation}
\beta_B:\operatorname{gr}^W H_B\xrightarrow{\sim}H_B
\label{eq:weight-framing}
\end{equation}
of its weight filtration.  If a graded piece has multiplicity, its
rational multiplicity space is retained; a matrix additionally
requires bases in these spaces.  Morphisms preserve the splittings.
A weight-framed extension of $\Q(a)$ by $\Q(b)$, $b>a$, is an
extension equipped with a rational section of its Betti quotient.
\end{definition}

This full splitting is stronger than a framing of the two extreme
weight-graded pieces.  Baer sum and rational scalar multiplication
give a vector space of framed extension classes.  There is an exact
sequence
\begin{equation}
\begin{aligned}
0\longrightarrow\operatorname{Hom}_{\Q}(\Q(a)_B,\Q(b)_B)
&\longrightarrow\operatorname{Ext}^1_{\mathrm{wf}}(\Q(a),\Q(b))\\
&\longrightarrow\operatorname{Ext}^1_{\mathrm{MHS}}(\Q(a),\Q(b))
\longrightarrow0.
\end{aligned}
\label{eq:weight-framed-exact-sequence}
\end{equation}
Equivalently, the framed extensions are the degree-one cohomology
of
\begin{equation}
\operatorname{Cone}\left(
R\operatorname{Hom}_{\mathrm{MHS}}(\Q(a),\Q(b))
\longrightarrow
R\operatorname{Hom}_{\Q}(\Q(a)_B,\Q(b)_B)
\right)[-1].
\label{eq:weight-framed-rhom}
\end{equation}
With standard Tate generators and the period convention fixed below,
this gives the chosen coordinate identification
\begin{equation}
\operatorname{Ext}^1_{\mathrm{wf}}(\Q(a),\Q(b))\simeq\C.
\label{eq:weight-framed-tate-extension}
\end{equation}
Formula \eqref{eq:weight-framed-rhom} computes relative extension
data; by itself it is not the morphism complex of a unital category.

We shall use period matrices in the graded de Rham space, consistently
with Section~2.  For a framed $H$ set
\begin{equation}
P_H=(\operatorname{gr}^W c_H)\,
\beta_{B,\C}^{-1}c_H^{-1}s_{H,\mathrm{dR}}^F.
\label{eq:framed-period-matrix}
\end{equation}
This is a unipotent automorphism of $\operatorname{gr}^W H_{\mathrm{dR}}$.
The factor $\operatorname{gr}^W c_H$ includes the standard Tate
periods $(2\pi i)^r$.  Thus \eqref{eq:framed-period-matrix}
fixes the scalar convention for \eqref{eq:weight-framed-tate-extension};
its off-diagonal entry is the actual framed period.

\subsection{Two \texorpdfstring{$A_\infty$}{A-infinity} functors}
\label{subsec:two-ainfty-functors}

Let $\mathcal C_X^{\mathrm{wf}}$ be the category of pairs
$(\mathbb V,\beta_B)$, where $\mathbb V\in\mathcal T(X)$ and
$\beta_B$ is a rational splitting of the weight filtration on
$\Phi_B(j_{**!}\mathbb V)\otimes\Q(1)$, with the fixed
nearby-cycle lift at $1$.  Morphisms preserve this splitting.
This category is abelian: exactness of the Betti realization and
strictness for weights induce splittings on kernels and cokernels.

Take the dg quotient of its bounded complexes by acyclic complexes
\cite{DrinfeldDG}, and let $\mathcal A_X^{\mathrm{wf}}$ be a
strictly unital minimal model of the full dg subcategory on the
constant Tate objects; see \cite[Sections~3 and~7]{KellerAinfty}.
Its reduced morphism spaces have strictly positive cohomological
degree: negative Ext groups vanish, and the only degree-zero maps
between Tate objects are scalar identities.  The weight filtration
represents every object by a finite extension-type twisted object,
hence by a finite Maurer--Cartan matrix $\delta$ satisfying
\begin{equation}
\sum_{r\geq1}m_r(\delta,\ldots,\delta)=0.
\label{eq:word-maurer-cartan}
\end{equation}
We use the standard strictly unital and minimal conventions, so the
operations have degree $2-r$.

The target Tate model $\mathcal T^{\mathrm{wf}}$ has objects
$T_a=\Q(a)$ and
\begin{equation}
\operatorname{Hom}^r(T_a,T_b)=
\begin{cases}
\Q\,\mathrm{id},&r=0,\ a=b,\\
\C,&r=1,\ b>a,\\
0,&\text{otherwise}.
\end{cases}
\label{eq:framed-target-category}
\end{equation}
The differential and the operations $m_r$ for $r\geq3$ are zero,
and products of two reduced degree-one arrows vanish.  In Betti weight
coordinates a framed Hodge--Tate structure is an arbitrary
unipotent complex period matrix; morphisms are graded rational
maps intertwining the matrices.  After choosing a $\Q$-basis of
$\C$, these data are finite-support representations of the quiver
with one arrow for each basis element from $a$ to $b$ whenever
$b>a$.  There are no relations, so this category is hereditary.
Thus the displayed dg category is a minimal Tate model for framed
Hodge--Tate structures.  We identify its finite extension-type
twisted objects with these structures by the period convention
\begin{equation}
\delta_H=P_H-1.
\label{eq:target-period-coordinate}
\end{equation}
This fixes the target coordinate used below.
We choose the target quasi-equivalence compatibly with
\eqref{eq:target-period-coordinate}, and take all subsequent
homotopy transfers with respect to this fixed choice.  Fixing
the coordinate only on two-weight extensions would not suffice
to determine it on longer extensions.

The two geometric realizations, with the common Tate twist,
are
\begin{align}
\mathbf H(\mathbb V)&=H^1(j_{**!}\mathbb V)\otimes\Q(1),
\label{eq:cohomology-functor}\\
\boldsymbol\Phi(\mathbb V)&=\Phi(j_{**!}\mathbb V)\otimes\Q(1).
\label{eq:vanishing-cycle-functor}
\end{align}
By the purity and shrinking results of Section~1, both realizations
preserve the basic Tate objects.  Thus, for every $a\in\mathbb Z$,
there are canonical identifications
\[
\mathbf H(\Q_X(a))\simeq T_a,
\qquad
\boldsymbol\Phi(\Q_X(a))\simeq T_a.
\]
For an object of $\mathcal C_X^{\mathrm{wf}}$, use its given
weight framing on the vanishing-cycle Betti realization and
transport it to cohomology by $\varphi_I$.  This makes the two output framings compatible
with shrinking.  On Tate objects the identifications above
respect these framings and the Tate-period convention.  Thus the
two realizations agree on Tate objects and identity morphisms;
their difference appears on extension data.  Both realizations are
exact functors to weight-framed mixed Hodge--Tate structures.
Applying them to bounded complexes and passing to the dg quotients
gives dg lifts.  Homotopy transfer therefore gives strictly unital
$A_\infty$ functors
\begin{equation}
\mathbf H_\infty,\boldsymbol\Phi_\infty:
\mathcal A_X^{\mathrm{wf}}\longrightarrow\mathcal T^{\mathrm{wf}},
\label{eq:two-ainfty-functors}
\end{equation}
where $\mathcal A_X^{\mathrm{wf}}$ is the chosen minimal Tate model.
Their higher components depend on the chosen transfer, but the
residue-coordinate expression below is the formal one relevant to the
Hodge defect calculation.

We now record that formal relative construction.  We use reduced
morphism spaces and send identities to identities.  For an adjacent
extension class
$x\in\operatorname{Hom}^1_{\mathcal A_X^{\mathrm{wf}}}(T_a,T_{a+1})$,
its de Rham class is
\[
\operatorname{res}_0(x)\,\omega_0+
\operatorname{res}_1(x)\,\omega_1,
\qquad \operatorname{res}_i(x)\in\Q.
\]
Here $\operatorname{res}_1$ is the coefficient of $\omega_1$
(the negative of the usual residue at $1$), in the Tate-period
convention of Section~1.  Constant extensions have zero residues.
Choose weight-framed lifts $\kappa_0,\kappa_1$ of the Kummer
classes represented by $\omega_0,\omega_1$, so that
\[
\operatorname{res}_i(\kappa_j)=\delta_{ij}.
\]
The residue formula is independent of the choice of these lifts.
Recall the numbers $g_w$ of
\eqref{eq:word-coefficient-formula}.  Send each Tate object to
$T_a$.  For a composable string of degree-one classes between
successive Tate degrees, define
\begin{equation}
(\mathcal G_\infty)_r(x_r,\ldots,x_1)
=\sum_{i_1,\ldots,i_r\in\{0,1\}}
g_{i_r\cdots i_1}\,
\operatorname{res}_{i_r}(x_r)\cdots
\operatorname{res}_{i_1}(x_1).
\label{eq:explicit-relative-components}
\end{equation}
The value lies in
$\operatorname{Hom}^1(T_a,T_{a+r})=\C$ of the target.
On every other string of reduced homogeneous morphisms put the
value equal to zero.  Finally, send identities to identities and
impose strict unitality.  The arguments in
\eqref{eq:explicit-relative-components} are written in composition
order, so that the rightmost arrow acts first.

The coefficients $g_w$ of the Hodge defect computed in Section~2
define the formal relative construction.  For a finite Maurer--Cartan
matrix $\delta$ representing $\mathbb V$, put
\begin{equation}
h=\sum_{r\geq1}(\mathbf H_\infty)_r(\delta,\ldots,\delta),
\qquad
p=\sum_{r\geq1}(\boldsymbol\Phi_\infty)_r(\delta,\ldots,\delta).
\label{eq:transferred-period-entries}
\end{equation}
Choose its twisted representative with the given identification of
the weight-graded pieces.  The block $\delta_{a+1,a}$ represents
the two-weight subquotient in Tate degrees $a,a+1$.  Consequently,
in the common graded de Rham frame,
\[
N_i|_{E_a}=\operatorname{res}_i(\delta_{a+1,a}),\qquad i=0,1,
\]
entrywise on multiplicity spaces.  Since $N_i$ raises Tate degree
by one, these blocks determine the entire word substitution.
With the target coordinate \eqref{eq:target-period-coordinate},
one has $1+h=P_{\mathbf H(\mathbb V)}$ and
$1+p=P_{\boldsymbol\Phi(\mathbb V)}$ in the common graded de Rham
frame.  The compatible Betti framings and
\eqref{eq:framed-period-matrix} therefore identify the matrix of
the Hodge defect with $(1+h)^{-1}(1+p)$.

\begin{rem}[Formal $A_\infty$ interpretation]
Formula \eqref{eq:explicit-relative-components}, together with strict
unitality, is the formal relative $A_\infty$-construction associated
with the Hodge defect.  In the chosen model it defines a strictly
unital functor
\[
\mathcal G_\infty:\mathcal A_X^{\mathrm{wf}}
\longrightarrow\mathcal T^{\mathrm{wf}}.
\]
The $A_\infty$ identities are formal: every reduced source operation
on reduced inputs has degree at least two, while reduced target
products vanish.  Evaluating on $\delta$ gives
\begin{equation}
\begin{aligned}
\mathcal G_I(\mathbb V)
&=1+\sum_{r\geq1}(\mathcal G_\infty)_r(\delta,\ldots,\delta)\\
&=(1+h)^{-1}(1+p).
\end{aligned}
\label{eq:higher-term-defect-entry}
\end{equation}
Indeed, substitution of the adjacent residue blocks gives the word expansion
of the matrix of the Hodge defect calculated in Section~2.  On Kummer classes this
reads
\[
(\mathcal G_\infty)_r(\kappa_{i_r},\ldots,\kappa_{i_1})
=g_{i_r\cdots i_1}.
\]
Thus the coefficients calculated in Section~2 are the components
of this formal relative comparison.  They contain the explicit
local factor $U(e_1)$; after its removal, the coefficients are
multiple zeta values as in Theorem~\ref{thm:defect-associator}.
The minimal models and homotopy transfers depend on choices;
\eqref{eq:explicit-relative-components} fixes the formal relative
comparison considered here.
\end{rem}

\bibliographystyle{alpha}
\bibliography{hodge}

\end{document}